\documentclass[a4paper]{amsart}
\usepackage{times}
\usepackage{cite}
\usepackage{amsthm,amsmath,amssymb,amsfonts}
\usepackage{algorithmic}
\usepackage{graphicx}
\usepackage{ hyperref}
\usepackage{textcomp}
\usepackage{mfl}
\usepackage{mathtools}
\usepackage{mathtools}
\usepackage{mathrsfs}
\usepackage{faktor}
\usepackage{fontawesome}

\usepackage{tikz}
\usetikzlibrary{arrows,chains,matrix,positioning,scopes}
\def\BibTeX{{\rm B\kern-.05em{\sc i\kern-.025em b}\kern-.08em
    T\kern-.1667em\lower.7ex\hbox{E}\kern-.125emX}}

\makeatletter
\tikzset{join/.code=\tikzset{after node path={%
\ifx\tikzchainprevious\pgfutil@empty\else(\tikzchainprevious)%
edge[every join]#1(\tikzchaincurrent)\fi}}}
\makeatother

\tikzset{>=stealth',every on chain/.append style={join},
         every join/.style={->}}
\tikzstyle{labeled}=[execute at begin node=$\scriptstyle,
   execute at end node=$]

\newtheorem{Thm}{Theorem}
\newtheorem{Rmk}[Thm]{Remark}
\newtheorem{Exm}[Thm]{Example}
\newtheorem{Cor}[Thm]{Corollary}
\newtheorem{Lem}[Thm]{Lemma}
\newtheorem{Pro}[Thm]{Proposition}

\newtheorem{Prob}{Problem}

\usepackage{color, dsfont}

\theoremstyle{plain}

\theoremstyle{definition}
\usepackage[]{graphicx}

\title{Maximality of  logic without identity}

\author{Guillermo Badia }
\address{
School of Historical and Philosophical Inquiry\\
 University of Queensland\\ 
 St Lucia, QLD 4072, Brisbane, Australia\\ 
\texttt{guillebadia89@gmail.com} }
\author{Xavier Caicedo}
\address{Departamento de Matem\'aticas\\ Universidad de los Andes \\ Carrera 1 N. 18 A -70\\ Bogot\'a, Colombia\\ \texttt{xcaicedo@uniandes.edu.co}}
\author{Carles Noguera }
\address{Institute of Information Theory and Automation\\
          Czech Academy of Sciences,
          Pod vod\'{a}-renskou v\v{e}\v{z}\'{i}~4, 182~00~Prague, Czech Republic\\
           \texttt{noguera@utia.cas.cz}}

\begin{document}

\date{}

\maketitle

\begin{abstract}
Lindstr\"om's theorem obviously fails as a characterization of $\mathcal{L}_{\omega \omega}^{-} $, first-order logic without identity. In this note we provide a fix: we show that  $\mathcal{L}_{\omega \omega}^{-} $ is a \emph{maximal}  abstract logic satisfying a weak form of the isomorphism property (suitable for identity-free languages and studied in~\cite{Casa}), the L\"owenheim--Skolem property, and compactness. Furthermore, we show that compactness can be replaced by being recursively enumerable for validity under certain conditions. In the proofs  we use a form of strong upwards L\"owenheim--Skolem theorem not available in the framework with identity. 

\bigskip


\noindent{\bf Math Subject Classification:}  03B10, 03C95

\end{abstract}

\section{Introduction}

In the 1960s, Per Lindstr\"om~\cite{lin} showed that first-order logic is maximal (in terms of expressive power) among its extensions satisfying certain combinations of model-theoretic results. The best known of these combinations are: 
\[
\text{L\"owenheim--Skolem theorem} \ +\ \text{Compactness}
\]
 \[
 \text{L\"owenheim--Skolem theorem} \ +\ \text{Recursively enumerable set of validities}
 \]
This list is by no means exhaustive though (the reader can consult the encyclopaedic monograph~\cite{barfer} for a thorough treatment of this topic). Philosophically, these results have been interpreted as providing a case for first-order logic being the ``right" logic in contrast to higher-order, infinitary, or logics with generalized quantifiers, which can be argued to be more mathematical beasts (see~\cite{tharp, ke}).  An implicit assumption of Lindstr\"om's work is that identity ($=$) is a most basic notion and belongs in the base logic. 

The classical Lindstr\"om theorems clearly fail for first-order logic without identity ($\mathcal{L}_{\omega \omega}^{-}$) since first-order logic with identity ($\mathcal{L}_{\omega \omega}$) is a proper extension of $\mathcal{L}_{\omega \omega}^{-}$. In fact, there are continuum-many logics between the former and the latter satisfying the compactness and L\"owenheim--Skolem properties, and with recursively enumerable sets of validities (see Example~\ref{1} below).

In this article we aim at finding a way to amend Lindstr\"om's two central theorems so that they apply in the identity-free context.\footnote{Recall that any criteria for first-order axiomatizability in terms of closure of a class of structures under certain algebraic operations can be recast as a Lindstr\"om-style theorem. In this way, \cite[Thm.~3.4]{Casa} can be seen as a Lindstr\"om-style result already in the literature for logic without identity.} Our proofs make heavy use of a property that is not available in the context with identity, namely, that we have an unrestricted upwards L\"owenheim--Skolem theorem applying even to finite models.  We also observe other maximality results: a very simple one for the monadic version of the logic (i.e.\ restricted to vocabularies that only have unary predicates), $\mathcal{L}_{\omega \omega}^{1-}$, as well as results for both $\mathcal{L}_{\infty \omega}^{-}$ and  $\mathcal{L}_{\omega \omega}^{-} $ in terms of a suitable variant of the Karp property. A simple byproduct will be a preservation theorem characterizing the identity-free fragment of first-order logic (we essentially obtain~\cite[Cor.~2.10]{Casa} by a rather different method).

$\mathcal{L}_{\omega \omega}^{-}$ has attracted mathematical attention in other works such as \cite{kei} where the problem of categoricity of theories in that logic is studied. Moreover, the results in the present paper may provide new insight on the philosophical discussion whether $\mathcal{L}_{\omega \omega}^{-}$ is suitable as a contender for the title of the ``right logic" against $\mathcal{L}_{\omega \omega}$. After all, the logicality of the $=$ predicate is not obvious (cf.~\cite{fe2}). So, if the criteria were to involve only indisputably logical operators (so no more than what $\mathcal{L}_{\omega \omega}$ already involves), be reasonably expressive (quite a bit can be formalized already in $\mathcal{L}_{\omega \omega}^{-}$, including set theory), and satisfy a neat Lindstr\"om-style characterization,  $\mathcal{L}_{\omega \omega}^{-}$ would appear to be as good an option as any. However, we will not pursue those issues here. 

We use the notion of an \emph{abstract logic} from~\cite[Def.\ II.1.1.1]{barfer} which presents logics as model-theoretic languages~\cite{fe} (see also~\cite{bar1,flum, lin}), not as consequence relations or collections of theorems. Furthermore, we assume logics to have the basic closure properties (including for the Boolean connectives) from~\cite[Def.\ II.1.2.1]{barfer}, except that in the atom property we use $\mathcal{L}_{\omega \omega}^{-} $ as the base logic. For greater generality, we do not require the \emph{relativization property}. As usual, if $\mathcal{L}$ and $\mathcal{L}'$ are logics, we write $\mathcal{L} \leq \mathcal{L'}$ if, for any vocabulary $\tau$ and any formula $\f \in \mathcal{L}(\tau)$, we can find an equivalent formula $\f' \in \mathcal{L'}(\tau)$.

$\mathcal{L}_{\omega \omega}^{-} $ is, properly speaking, a fragment of $\mathcal{L}_{\omega \omega}$ containing the guarded fragments corresponding to basic modal logics. In this setting, the most fruitful approach has been to use bisimulations  as a modal analogue of potential isomorphisms in first-order logic~\cite{van}. In the present context all we require is the notion of weak (partial) isomorphism introduced in~\cite{Casa}, which is stronger than bisimulation.\footnote{This notion has incidentally proven useful in recent philosophical debates on the logicality of quantifiers and other operators~\cite{bo, Casa2}.}
 
Interestingly, the presence of identity can make a substantial difference regarding compactness. For example, monadic first-order logic with the Henkin quantifier, $\mathcal{L}^1_{\omega \omega}(Q^H)$,  is not compact and not contained in (monadic) first-order logic with identity for it can express the quantifier ``there are at least $\aleph_0$-many elements"; however, the identity-free fragment of the very same logic admits the effective elimination of the quantifier $Q^H$ and, hence, it is compact~\cite[Thm.\ 1.5]{kry}.\footnote{In contrast, the logic obtained from (monadic) identity-free first-order logic by adding the quantifier ``there are at least $\aleph_0$ elements" does not satisfy compactness~\cite[Thm.\ 8]{ya}.}  

The paper is arranged as follows: in \S\ref{pre} we start with the preliminary observation that there is a continuum of abstract logics between $\mathcal{L}_{\omega \omega}^{-} $ and $\mathcal{L}_{\omega \omega}$, and we recall the definitions of the properties of abstract logics employed in the paper, while referring to the literature for some particular technical notions. In \S \ref{main} we present our main new results, that is, Lindstr\"om-style characterizations of the identity-free first-order logic and its monadic fragment, together with instrumental observations regarding the logical relations of the involved properties and a useful form of upwards L\"owenheim--Skolem theorem. In \S \ref{ext} we examine a few interesting particular extensions of $\mathcal{L}_{\omega \omega}^{-} $ that help us understand the role of compactness and the L\"owenheim--Skolem property  in our characterizations. Finally, in \S \ref{con}, we collect some open problems that arise from this investigation.

\section{Preliminaries}\label{pre}

We begin this section by noting that there are continuum-many pairwise non-equivalent abstract logics between $\mathcal{L}_{\omega \omega}^{-} $ and $\mathcal{L}_{\omega \omega}$ (actually, already between their monadic fragments). 

\begin{Exm}\label{1}\em
Consider quantifiers $\exists^{\geq n}$ with semantics $\model{A} \models \Exi{^{\geq n}x} \f$ iff there are at least $n$ elements $a$ such that $\model{A} \models \f[a]$. 
For each non-empty $X \subsetneq \omega
\setminus \{0,1\}$, we can prove that the logic $\mathcal{L}_{\omega \omega}^{-} (\{\exists^{\geq n} \mid n\in X \})$ indeed lies properly between $\mathcal{L}_{\omega \omega}^{-} $ and $\mathcal{L}_{\omega \omega}$ in terms of expressive power and, moreover, there is a continuum of such intermediate abstract logics.

For distinct $X, Y \subseteq \omega
\setminus \{0,1\}$, the corresponding  logics $\mathcal{L}_{\omega \omega}^{-} (\{\exists^{\geq n} \mid n\in X \})$ and $ \mathcal{L}_{\omega \omega}^{-} (\{\exists^{\geq n} \mid n\in Y \})$ are also distinct. To see this, it suffices to focus our attention on a monadic vocabulary $\tau=\{P\}$. Suppose, without loss of generality, that we have an element $r \in X \setminus Y$. We abbreviate, for $n < m$, $\Exi{^{\geq n}x}\theta \wedge \lnot \Exi{^{\geq m}x}\theta $ as $\exists ^{\lbrack n,m)}x\,\theta $, and, for each $n$, $\exists ^{\geq n}x\,\theta $ as $\exists ^{\lbrack n,\infty )}x\,\theta$. Then, using results  from~\cite{Cai}, any sentence  $\varphi$ from the logic  $\mathcal{L}_{\omega \omega}^{-} (\{\exists^{\geq n} \mid n\in Y \})$ over the vocabulary $\tau$ is equivalent to a disjunction $\theta_1 \vee \dots \vee\theta_q$ involving only quantifiers from $\varphi$ where each $\theta _{i}$ is of one the three following forms:
\begin{itemize}
\item $\exists ^{\lbrack n_{i},m_{i})}x\, P(x)\wedge \exists ^{\lbrack
r_{i},s_{i})}x\,\lnot P(x)$
\item $\exists ^{\lbrack n_{i},m_{i})}x\, P(x)$
\item $\exists ^{\lbrack r_{i},s_{i})}x\,\lnot P(x)$
\end{itemize}
where $n_{i} \leq m_{i}$ and $r_{i} \leq s_{i}$ belong to $Y\cup \{1,\infty \}.$
Thus, $\varphi $ just describes an array of possible cardinalities for the
interpretations of $P$ and its complement, and clearly, $\Exi{^{\geq r}x}P(x)$ is equivalent to this disjunction if and only if $[r,\infty )=\bigcup 
_{n_i < m_i}[n_{i},m_{i})$, or $r=n_{i}$ for the least $n_{i},$ which is impossible
as $r\not\in Y\cup \{1\}.$
\end{Exm}

We use the definitions from~\cite{Casa}: $\model{A} \sim \model{B} $ means that there is a \emph{relativeness correspondence} between the structures~\cite[Def.\ 2.5]{Casa} (we prefer to call this a \emph{weak isomorphism});  $\model{A}\sim_p\model{B}$ means that there is a back-and-forth system $I$ of \emph{partial relativeness correspondences} between the models~\cite[Def.\ 4.7]{Casa} (we can say that these structures are \emph{partially weakly isomorphic}); and we denote by $\sim_n$ the finite approximation of  $\sim_p$~\cite[Def.\ 4.2]{Casa}. In the setting of first-order logic without identity, the relation $ \sim$ behaves like a weak notion of isomorphism~\cite{Casa}, which motivates the name for the second property defined below.\footnote{ Another place in the literature where this has been studied, albeit in less detail, is~\cite{ur}.}

The properties of abstract logics that we consider in this article are:

\begin{itemize}
\item \emph{Compactness property}:  for any vocabulary $\tau$, $\Phi \subseteq \mathcal{L}(\tau)$, if every finite subset of $\Phi $ has a model then $\Phi$ has a model.
\item \emph{L\"owenheim--Skolem property}: for any vocabulary $\tau$, and sentence $\f \in \mathcal{L}(\tau)$, $\f$ has a countable model if it has an infinite model.
\item \emph{Weak isomorphism property}: for any structures $\model{A}$ and $ \model{B}$, $\model{A} \sim \model{B}  $ only if $  \model{A} \equiv_{\mathcal{L}} \model{B}$.
\item \emph{Finite weak dependence property}: for any vocabulary $\tau$ and any $\varphi \in \mathcal{L}(\tau)$, there is a finite $\tau_0\subseteq \tau$ s.t.\ for any $\tau$-structures $\model{A}$ and\/ $\model{B}$, if\/ $\model{A}\upharpoonright  \tau_0\sim\model{B} \upharpoonright\tau_0$, then $\model{A}\models \varphi $ iff\/ $\model{B} \models \varphi$.
\item \emph{Karp$^-$ property}: for any structures $\model{A}$ and $ \model{B}$, $\model{A} \sim_p \model{B}$ only if $\model{A} \equiv_{\mathcal{L}} \model{B}$.
\item \emph{Boundedness property}: any sentence $\f(<, \dots)$ which for arbitrary large ordinal type $\alpha$ has a model where the interpretation of $<$ is an irreflexive and transitive binary relation containing a chain of order type $\alpha$  has a model where the interpretation of $<$ contains an infinite descending chain.
\end{itemize}
All these properties, with the exception of Karp$^-$ and weak isomorphism, hold in $\mathcal{L}_{\omega \omega}.$

Given a structure $\model{A}$, we denote by $\model{A}^*$ the \emph{reduction} of $\model{A}$~\cite[Def.\ 2.4]{Casa}, i.e., the quotient structure $\faktor{\model{A}}{\Omega(\model{A})}$ obtained from the Leibniz congruence relation.

\begin{Pro}[\cite{Casa}]\label{pro1} Let\/ $\model{A}$ and\/ $\model{B}$ be structures. Then:
\begin{itemize}
\item[(i)] If\/ $\model{A}$ and\/ $ \model{B}$ are countable, then $\model{A} \sim_p \model{B}$ iff\/ $\model{A} \sim \model{B}$.
\item[(ii)] $\model{A} \sim \model{B}$ iff\/ $\model{A}^* \cong \model{B}^*$.
\end{itemize}
\end{Pro}

Thanks to Proposition~\ref{pro1}, the \emph{weak isomorphism property} can be equivalently formulated as follows: for any structures $\model{A}$ and $ \model{B}$, $\model{A}^* \cong \model{B}^*$ only if $\model{A} \equiv_{\mathcal{L}} \model{B}$. 
Observe that $\model{A}$ and $\model{A}^*$ are relatives: $\model{A}^*\cong \model{A}^{**}$, so by Proposition~\ref{pro1}, $\model{A} \sim \model{A}^*$.

\section{Maximality results}\label{main}
We start this section by showing a form of \emph{upwards L\"owenheim--Skolem theorem}, which will be heavily used in the arguments below:

\begin{Lem}\label{use}
Let $\mathcal{L}$ be an abstract logic with the weak isomorphism property. Then,  a theory $T \subseteq \mathcal{L}(\tau)$ has a model $\model{A}$ of cardinality $\lambda$ only if, for any $\kappa>\lambda$,  there is a model $\model{B}$ of T with cardinality $\kappa$ and a surjective strict homomorphism (in the sense of~\cite[Def.\ 2.1]{Casa}), and hence a weak isomorphism, from $\model{B}$ onto $\model{A}$.
\end{Lem}

\begin{proof} It follows by inspection of the proof of~\cite[Lem.\ 2.24]{bridge} or \cite[Ch. IV, \S 1]{ack} (which is only formulated for relational languages but can be easily generalized to languages with function symbols). For any structure $\model{A}$ of cardinality $\lambda$, in that proof one builds a model $\model{B}$ of size $\kappa$ and a mapping $B \longrightarrow A$ which is, in fact, a surjective strict homomorphism.  
\end{proof}

\begin{Rmk} \emph{Lemma~\ref{use} allows us to see that a plethora of logics do not have the weak isomorphism property, e.g.\ the logics in Example~\ref{1}. Interestingly, the usual Lindstr\"om quantifiers destroy the property, in particular in the logics $\mathcal{L}_{\omega \omega}^{-}(Q_\alpha)$. However, as we will see in Example~\ref{ex2},  all of these logics have counterparts which do have the weak isomorphism property.  On the other hand, as we will see below, the Henkin quantifier $Q^H$ is a curious case of a natural  Lindstr\"om quantifier that has the weak isomorphism property.}
\end{Rmk}

Now we can provide an analogue of (1) from~\cite[Thm.\ III.1.1.1]{barfer}.
%

\begin{Lem}\label{fdp}
Let $\mathcal{L}$ be an abstract logic such that $\mathcal{L}_{\omega \omega}^{-} \leq \mathcal{L}$. If $\mathcal{L}$ has the compactness and  weak isomorphism properties, then it also has the \emph{finite weak dependence property}.
\end{Lem}

\begin{proof}
Given a vocabulary $\tau ,$ let $\tau ^{\prime }$ be a disjoint copy and
consider the theory $\Phi (\tau ,R)$:
\begin{eqnarray*}
&\{\forall x_{1}\dots x_{n}\forall y_{1}\dots y_{n}[\bigwedge
_{i}Rx_{i}y_{i}\rightarrow (\theta (x_{1}\dots)\leftrightarrow \theta ^{\prime
}(y_{1}\dots))]\mid \theta \in \tau \text{, }\theta ^{\prime }\in \tau ^{\prime }%
\text{ its copy} \}\\
&\cup \{ \All{ x_{1}, \dots, x_{n}}\,\All {y_{1}, \dots,y_{n}} [\bigwedge _{i}Rx_{i}y_{i} \rightarrow Rt(x_1\dots)t'(y_{1} \dots)]   \mid  t \ \text{a term of} \  \tau\}
\\ &\cup \  \{``R\text{ and }R^{-1}\text{
are surjective''}\}
\end{eqnarray*}%
For any $\varphi \in \mathcal{L(\tau )}$, let $\varphi ^{\prime }$ denote its
renaming in the type $\tau ^{\prime }$. Then, $\Phi (\tau ,R)\models \varphi
\leftrightarrow \varphi ^{\prime }$ by closure of the logic $\mathcal{L}$
under weak isomorphisms, and by compactness 
\begin{equation*}
\Phi (\tau _{0},R)\models \varphi \leftrightarrow \varphi ^{\prime }
\end{equation*}%
for some finite $\tau _{0}\subseteq \tau $.

Assume now that $\mathfrak{A}\upharpoonright \tau _{0}\sim \mathfrak{B} \upharpoonright \tau _{0}$ by some $\tau _{0}$-weak isomorphism $r\subseteq (A\cup B)^{2},$ and $|A|<|B|.$ By Lemma~\ref{use}, there is a $\mathfrak{C}$ of power $
|B|$ and a surjective strict homomorphism $h:C\rightarrow A$. Thus, $r\circ h$ is a $\tau _{0}$-weak isomorphism from $\mathfrak{C}$ onto $\mathfrak{B}.$ Renaming the last structure as $\mathfrak{B}^{\prime }$ with $\tau ^{\prime}$ we may put $\mathfrak{C}$ and $\mathfrak{B}^{\prime }$ together in a structure $\mathfrak{C}+\mathfrak{B}^{\prime }$ sharing the same domain. Then, $\tuple{\mathfrak{C}+\mathfrak{B}^{\prime},r\circ h}\models \Phi (\tau _{0},R),$ and hence, $\tuple{\mathfrak{C}+\mathfrak{B}^{\prime}, r\circ h}\models \varphi \leftrightarrow \varphi ^{\prime }$ this implies: $
\mathfrak{C}\models \varphi \ $\ iff $\mathfrak{B}\models \varphi .$ But $
\mathfrak{A}\sim \mathfrak{C}$ with respect to full $\tau ,$ then $\mathfrak{A}\models \varphi \ $\ iff $\mathfrak{B}\models \varphi.$ If $|A|=|B|$, we apply the construction directly with $\mathfrak{A}$ and $\mathfrak{B}$. 
\end{proof}

We are now ready to provide the main result of this paper:

\begin{Thm}\label{thm1} Let $\mathcal{L}$ be an abstract logic such that $\mathcal{L}_{\omega \omega}^{-} \leq \mathcal{L}$. If $\mathcal{L}$ has the weak isomorphism, compactness, and L\"owenheim--Skolem properties, then $\mathcal{L} \leq  \mathcal{L}_{\omega \omega}^{-}$.
\end{Thm}

\begin{proof}
Assume $\varphi\in\mathcal{L}(\tau)\setminus\mathcal{L}_{\omega\omega}^{-}(\tau)$ and $\varphi$ depends on a finite vocabulary $\tau_{0}\subseteq\tau$ (by compactness and Lemma~\ref{fdp}). Notice that there are only finitely many sentences of rank $\leq n$ in $\mathcal{L}_{\omega\omega}^{-}(\tau_{0})$ \cite[Lem.\ 4.4]{Casa}, thus the relation
$\mathfrak{A}\upharpoonright\tau_{0}\equiv_{n}^{-}\mathfrak{B}\upharpoonright
\tau_{0}$ has finitely many equivalence classes of structures of type $\tau$
and the equivalence class of a structure $\mathfrak{A}$ coincides with
$\mathrm{Mod}_{\tau}(\Theta_{\mathfrak{A}})$ for the sentence
\[
\Theta_{\mathfrak{A}}=\bigwedge\{\theta\text{\ of\ rank}\leq n\mid
\mathfrak{A\models}\theta\}.
\]
Therefore, $\mathrm{Mod}_{\tau}(\varphi)$ cannot be a union of these classes (it would be equivalent to a finite disjunction of sentences in $\mathcal{L}_{\omega\omega}^{-}(\tau_{0}))$ and it must cut some equivalence class in two non-emtpy pieces. In other words, there are $\tau$-structures $\mathfrak{A}_{n}$ and $\mathfrak{B}_{n}$ such that
\[
\mathfrak{A}_{n}\upharpoonright\tau_{0}\equiv_{n}^{-}\mathfrak{B}_{n}\upharpoonright\tau_{0}\emph{,\ }\text{\ }\mathfrak{A}_{n}\models \varphi,\mathfrak{B}_{n}\models\lnot\varphi.
\]
By Lemma~\ref{use}, we may assume that $\mathfrak{A}_{n}\ $and $\mathfrak{B}_{n}$ have the same infinite power and share the same domain $A_{n}.$

By~\cite[Lem.\ 4.4 and Prop.\ 4.5]{Casa}, $\mathfrak{A}_{n}\upharpoonright
\tau_{0}\sim_{n}\mathfrak{B}_{n}\upharpoonright\tau_{0};$ that is, there are
sets $I_{0},\ldots,I_{n}$ of weak \emph{finite} $\tau_{0}$-partial
isomorphisms from $\mathfrak{A}_{n}$ to $\mathfrak{B}_{n}$ such that
$I_{n}\not =\emptyset$ and for all $p\in I_{j+1}$ $a,\in A_{n},b\in B_{n}$
there are $q,q^{\prime}\in I_{j}$ such that $q,q^{\prime}\supseteq r$ and
$a\in\mathit{dom}(q),$ $b\in\mathit{rg}(q^{\prime}),$ and further
extension properties guaranteeing that constants and functions are eventually
preserved. The set of all finite weak $\tau_{0}$-partial isomorphisms has
the same power as $D,$ so we may enumerate them as $\{R_{p}\mid p\in A_{n}\};$
moreover, we may assume $\{0,\ldots,n\}\subseteq A_{n}.$ Then, renaming
$\mathfrak{B}_{n}$ as $\mathfrak{B}_{n}^{\prime}$ on the vocabulary
$\tau^{\prime},$ as in the proof of Lemma~\ref{fdp}, and defining in $A_{n}$:

\medskip

$<^{\ast}=$ usual order of $\{0,\ldots,n\}$

$c_{0}^{\ast}=n$

$\tuple{j,p}\in I^{\ast}\Leftrightarrow R_{p}\in I_{j}$

$\tuple{p,x,y}\in G^{\ast}\Leftrightarrow \tuple{x,y}\in R_{p},$\medskip

\noindent the structure $\tuple{\mathfrak{A}_{n}+\mathfrak{B}_{n},<^{\ast},c_{0}^{\ast},I^{\ast},G^{\ast}}$ satisfies the following finite theory $\Psi$ in the vocabulary
\[
\tau_{0}\cup\tau_{0}^{\prime}\cup\{<,c_{0},I,G\},
\]
where $c_{0}$ is a constant, $<$ and $I$ are binary relations, and $G$ is a
ternary relation (all fresh symbols; moreover, for each formula $\psi
\in\mathcal{L}(\tau_{0})$, we denote by $\psi^{\prime}$ its renaming in the
vocabulary $\tau_{0}^{\prime}$):\medskip

\begin{tabular}
[c]{ll}%
1. & $\varphi$, $\lnot\varphi^{\prime}$\\
2. & $\exists p$ $Ic_{0}p$\\
3. & $\forall p\,\overrightarrow{x}\overrightarrow{y}(%
{\textstyle\bigwedge\nolimits_{i}}
Gpx_{i}y_{i}\rightarrow(\chi(\overrightarrow{x})\leftrightarrow\chi^{\prime
}(\overrightarrow{y}))),$\\
& \ \ \ \ for each relation symbol$\chi\in\tau_{0}$ of arity $|\overrightarrow
{x}|$\\
4. & $\forall uvp\overrightarrow{x}\overrightarrow{y}(u<v\wedge Ivp\wedge%
{\textstyle\bigwedge\nolimits_{i}}
Gpx_{i}y_{i}\rightarrow\exists q[Iuq\wedge Gqf(\overrightarrow{x})f^{\prime
}(\overrightarrow{y})$\\
& $\wedge\forall zw(Gpzw\rightarrow Gqzw)])$,\\
& \ \ \ \ \ for each function symbol $f\in\tau_{0}$ of arity $|\overrightarrow
{x}|$\\
5. & $\forall uvp(u<v\wedge Ivp\rightarrow\exists q[Iuq\wedge Gqcc^{\prime
}\wedge\forall zw(Gpxw\rightarrow Gqzw)])$,\\
& \ \ \ \ \ for each constant symbol $c\in\tau_{0}$\\
6. & $\forall uvp(u<v\wedge Ivp\rightarrow\forall x\exists qq^{\prime
}yy^{\prime}[Iuq\wedge Iuq^{\prime}\wedge Gqxy\wedge Gq^{\prime}y^{\prime}x$\\
&
$ \wedge
\forall zw(Gpxw\rightarrow Gqzw\wedge Gq^{\prime}zw)]$
\end{tabular}
\medskip

\noindent The second sentence states that $I_{n}$ is non-empty. Sentences 3--6 describe a sequence $I_{0},\ldots,I_{n}$ of sets of weak $\tau_{0}$-partial isomorphisms in the sense  of \cite[Def.\ 4.2]{Casa}.


As the above holds for any $n$, we have models for any finite part of the infinite theory with additional constants $c_{1},c_{2},...$:
\[
\Psi(\tau_{0},<,I,G)\cup\{\varphi,\lnot\varphi^{\prime}\}\cup\{c_{j+1}<c_{j}\mid j\in\omega\}.
\]
By compactness, we have a model
$\tuple{\mathfrak{C},<^{\mathfrak{C}},I^{\mathfrak{C}
},G^{\mathfrak{C}},\tuple{c_{j}^\mathfrak{C}}_{j\in\omega}}$ of this theory.
By the axioms, each $p\in C$ encodes a weak $\tau_{0}$-partial isomorphism
$R_{p}=\{\tuple{x,y}\in A^{2}\mid\tuple{p,x,y}\in G^{\mathfrak{C}}\}\ $between
$\mathfrak{C}\upharpoonright\tau_{0}\ $and $\mathfrak{C}\upharpoonright
\tau_{0}^{\prime},$ and the sequence
\[
I_{j}=\{R_{p}\in C\mid\tuple{p,c_{j}^{\mathfrak{C}}}\in I^{\mathfrak{C}}\}, j=0,1,\ldots
\]
has the back-and-forth extension property with respect to increasing
subindexes: if $R_{p}\in I_{j}$ and $c\in C$, then there is a $R_{q}\in
I_{j+1}$ such that $c\in\mathit{dom}(R_{p}),$ etc. Hence, $K^{\mathfrak{C}}=\bigcup_{j}I_{j}$ has the unrestricted extension property and becomes a Karp
system of weak $\tau_{0}$-isomorphisms. This is expressible by the finite
theory $\Phi(\tau_{0},K,G)$ which results of changing the back-and-forth
axioms of $\Psi(\tau_{0},<,c_{0},I,G)$ to
\[
\forall px(Kp\rightarrow\exists q\,q^{\prime}\,\exists yy^{\prime}[Kq\wedge
Kq^{\prime}\wedge Gqxy\wedge Gq^{\prime}y^{\prime}x\wedge\forall z\forall
w(Gpzw\rightarrow Gqzw\wedge Gq^{\prime}zw)]
\]
In sum, $\tuple{\mathfrak{C},K^{\mathfrak{C}},G^{\mathfrak{C}}}\models
\Phi(\tau_{0},K,G)\cup\{\varphi,\lnot\varphi\}$ which means
\[
\mathfrak{C}\upharpoonright\tau_{0}\sim_{p}\mathfrak{C}\upharpoonright\tau
_{0}^{\prime},\text{ }\mathfrak{C}\upharpoonright\tau\models\varphi
,\mathfrak{C}\upharpoonright\tau^{\prime}\models\lnot\varphi^{\prime}.
\]
By the L\"{o}wenheim--Skolem property, we may assume that $\mathfrak{C}$ is
countable. Hence, by Proposition~\ref{pro1}, $\mathfrak{C\upharpoonright\tau
}_{0}\sim\mathfrak{C}\upharpoonright\tau_{0}^{\prime}$ and thus
$\mathfrak{C\mathfrak{\upharpoonright\tau}\models\varphi\Longleftrightarrow
C\mathfrak{\upharpoonright\tau}}^{\prime}\mathfrak{\models\varphi}$ by the
choice of $\tau,$ a contradiction.
\end{proof}

\begin{Rmk} \emph{The Karp$^-$ property may replace the  L\"{o}wenheim--Skolem hypothesis in the above theorem because the proof yields before the last step a model of $\Phi(\tau_{0},K,G)\cup\{\varphi,\lnot\varphi\}$ for any finite $\tau_{0}\subseteq\tau,$ which by an additional use of compactness gives a model $\tuple{\mathfrak{C},K^{\mathfrak{C}},G^{\mathfrak{C}}}$ of $\Phi(\tau,K,G)\cup\{\varphi,\lnot\varphi\};$ that is, the weak isomorphisms encoded by $K,G$ are weak $\tau$-isomorphisms, thus we have 
\[
\mathfrak{C\upharpoonright\tau\sim}_{p}\text{ }\mathfrak{C\upharpoonright\tau
}^{\prime},\text{ }\mathfrak{C\mathfrak{\upharpoonright\tau}\models\varphi
,}\text{ }\mathfrak{C\mathfrak{\upharpoonright\tau}}^{\prime}\mathfrak{\models
\lnot\varphi}^{\prime}
\]
which, by the Karp$^-$ property, gives directly the contradiction
$\mathfrak{C\mathfrak{\upharpoonright\tau}\models\varphi\Longleftrightarrow
\mathfrak{C\mathfrak{\upharpoonright\tau}}^{\prime}\models\varphi^{\prime}}$.}
\end{Rmk}

\begin{Rmk}\emph{
Note that the boundedness property for $\mathcal{L}_{\infty \omega}^{-}$ is essentially a corollary of the classical one from~\cite[Thm.~1.8]{bar2}. Then, if we use our approach in encoding weak partial isomorphisms in Theorem~\ref{thm1} and working with the Karp$^-$ property, it is straightforward to modify the argument from~\cite[Thm.~III.3.1]{barfer} to show that  $\mathcal{L}_{\infty \omega}^{-}$  is maximal among its extensions in having the  boundedness, and Karp$^-$ properties. In fact, all we need from the boundedness property is that it will give us a model where $<$ is not well founded.}
\end{Rmk}

Comparing the proof of Theorem~\ref{thm1} with that of its classical counterpart with identity, the reader should note that our approach makes a substantial use of the strong upwards L\"owenheim--Skolem theorem given by Lemma~\ref{use}.  This allows us to deal with cardinality situations that in the classical context are dealt with the expressive power of identity. 

One may wonder whether we can obtain a Lindstr\"om-style characterization for identity-free monadic first-order logic, $\mathcal{L}_{\omega \omega}^{1-}$, analogous to Tharp's result~\cite[Thm.\ 1]{tharp} for {\em monadic} first-order logic. The answer is yes and the result does not require, surprisingly, any form of the L\"{o}wenheim--Skolem theorem (not even the other two properties if we assume the finite weak dependence property; see Remark~\ref{rem:Finite-Weak-Dependence}).

\begin{Thm}\label{mon}
Let $\mathcal{L}$ be a monadic logic such that $\mathcal{L}_{\omega\omega}^{1-}\leq\mathcal{L}.$ If $\mathcal{L}$ satisfies the compactness and weak isomorphism properties, then $\mathcal{L}\leq\mathcal{L}_{\omega\omega}^{1-}$.\end{Thm}
 
\begin{proof}
Assume $\varphi\in\mathcal{L}(\tau)\setminus
\mathcal{L}_{\omega\omega}^{-}(\tau)$, $\tau=\{P_{i} \mid i\in I\}$. As in the
proof of Theorem~\ref{thm1}, we have for each finite $\tau_{0}\subseteq\tau$:
\[
\mathfrak{A}\upharpoonright\tau_{0}\equiv_{1}^{-}\mathfrak{B}\upharpoonright
\tau_{0}\emph{,\ }\text{\ }\mathfrak{A}\models\varphi,\text{ }\models
\lnot\varphi.
\]
and by compactness
\[
\mathfrak{A}\equiv_{1}^{-}\mathfrak{B}\emph{,\ }\text{\ }\mathfrak{A}\models\varphi,\mathfrak{B}\models\lnot\varphi.
\]
By Lemma 3, we may assume $\mathfrak{A}\ $and $\mathfrak{B}$ share the same
domain $A.$

Each map $\delta \colon I\rightarrow\{0,1\}$ determines a type
\[
t_{\delta}(x)=\{P_{i}(x) \mid \delta(i)=1\}\cup\{\lnot P_{i}(x) \mid \delta(i)=0\}.
\]
A type $t_{\delta}$ is consistent with $\mathfrak{A}$ if for each finite
$J\subseteq I,$ $\mathfrak{A}\models\exists x\wedge(t_{\delta}(x)\upharpoonright J)$. Clearly, $\mathfrak{A}$ and $\mathfrak{B}$ above have
the same consistent types and, if $t_{\delta}$ is not consistent with
$\mathfrak{A}$, there is a witness $\eta_{\delta}$ of the form $\lnot\exists
x\wedge(t_{\delta\upharpoonright J_{\delta}}(x)),$ $J_{\delta}\subseteq
_{fin}I,$ true in both $\mathfrak{A}\ $and $\mathfrak{B}$.

Consider the following theory on the vocabulary $\tau\cup\tau^{\prime}\cup\{P_{\delta
},P_{\delta}^{\prime}\mid\delta\in2^{I}\}:$

\begin{tabular}{ll}
$-$ & $\varphi,\lnot\varphi^{\prime}$\medskip\\

 & For each $t_{\delta}$ consistent with $\mathfrak{A}$ and each finite $J\subseteq I$:\\

$-$ &  $\exists xP_{\delta}(x),$ $\forall x(P_{\delta}(x)\rightarrow\wedge(t_{\delta\upharpoonright J}(x))$\\

$-$ & $\exists xP_{\delta}^{\prime}(x)$, $\forall x(P_{\delta}^{\prime}(x)\rightarrow\wedge(t_{\delta\upharpoonright J}^{\prime
}(x))$.\medskip\\

&  For each $t_{\delta}$ inconsistent with $\mathfrak{A}$:\\

$-$ &  $\eta_{\delta},$ $\eta_{\delta}^{\prime}.$

\end{tabular}

Then, $\mathfrak{C=A}+\mathfrak{B}^{\prime}$ may be expanded to a
model $\tuple{\mathfrak{A}+\mathfrak{B}^{\prime}, P_{\delta}^{\mathfrak{C}
},P_{\delta}^{\prime\mathfrak{C}}}_{\delta\in 2^{J}}$ of each finite
part\ $\Sigma$ of this theory taking $P_{\delta}^{\mathfrak{C}}=\{a\in
A\mid \mathfrak{A\models}t_{\delta\upharpoonright J}(a)\}$ and $P_{\delta}^{\prime\mathfrak{C}}=\{b\in A\mid \mathfrak{B}^{\prime}\mathfrak{\models}t_{\delta\upharpoonright J}^{\prime}(a)\}$ for $J=\{i \mid P_{i}$ or
$P_{i}^{\prime}$ occur in $\Sigma\}$.

By compactness, there is a model $\tuple{\widehat{\mathfrak{A}}+\widehat
{\mathfrak{B}}^{\prime},P_{\delta}^{\mathfrak{A}},P_{\delta}^{\prime
\mathfrak{B}^{\prime}}}_{\delta\in2^{J}}$ of the full theory. Then,
$\widehat{\mathfrak{A}}$ and $\widehat{\mathfrak{B}}$ realize exactly the same
types $t_{\delta}$ (those originally consistent)\ and thus $\widehat
{\mathfrak{A}}\sim\widehat{\mathfrak{B}},$ defining $aRb$ iff $a$ and $b$
realize the same type $t_{\delta}.$ This contradicts the weak isomorphism
property since $\widehat{\mathfrak{A}}\models\varphi$ and $\widehat{\mathfrak{B}}\models\lnot\varphi.$
\end{proof}

\begin{Rmk}\label{rem:Finite-Weak-Dependence}
\emph{If $\mathcal{L}$ has the finite weak dependence property, then the compactness and weak isomorphism properties are not needed in the previous theorem. Indeed, if $\varphi$ depends on finite $\tau_{0}\subseteq\tau$, the first step of the proof $\mathfrak{A}\upharpoonright
\tau_{0}\equiv_{1}^{-}\mathfrak{B}\upharpoonright\tau_{0},$ $\mathfrak{A}\models\varphi,$ $\mathfrak{B}\models\lnot\varphi,$ yields already a contradiction, since $\mathfrak{A}\ $and $\mathfrak{B}$ realize trivially
the same $t_{\delta\text{ }}$types based on $\tau_{0},$ and thus
$\mathfrak{A}\upharpoonright\tau_{0}\sim\mathfrak{B}\upharpoonright\tau_{0}.$}
\end{Rmk}

Since  $\mathcal{L}_{\omega \omega}$ and $\mathcal{L}^1_{\omega \omega}$ have both the compactness and the L\"owenheim--Skolem  properties, then we can obtain the following preservation result from Theorems~\ref{thm1} and~\ref{mon} (which is essentially~\cite[Cor.\ 2.10]{Casa}\footnote{Note that~\cite[Cor.\ 2.10]{Casa} is equivalent to our formulation due to~\cite[Pro.\ 2.6]{Casa}.} proved by a rather different method): 

\begin{Cor} $\mathcal{L}_{\omega \omega}^{-}$ (resp.\ $\mathcal{L}_{\omega \omega}^{1-}$) is the fragment of $\mathcal{L}_{\omega \omega}$ (resp.\ $\mathcal{L}^1_{\omega \omega}$) preserved under weak isomorphisms.
\end{Cor}

We proceed now to obtain an analogue of the second Lindstr\"om theorem from~\cite{lin}. First, we need the following lemma:

\begin{Lem}\label{rec}\footnote{This lemma is an analogue of  \cite[Lem.\ III.1.1.2]{barfer} for $\mathcal{L}_{\omega \omega}$, but simpler.  In particular, we do not need to use the L\"owenheim--Skolem property.}
Let $\mathcal{L}$ be an abstract logic such that $\mathcal{L}_{\omega\omega}^-\leq \mathcal{L}$ satisfying the finite weak dependence and weak isomorphism properties. If $\mathcal{L}$ extends properly $\mathcal{L}_{\omega\omega}^-$,  then there exist a finite vocabulary $\sigma$ containing at least one unary relation $U$ and, for each finite vocabulary
$\rho\supseteq\sigma$, a sentence $\theta\in \mathcal{L}(\rho)$ such that

\begin{enumerate}
\item for each $n\geq1$, there is a model\/ $\model{A}\models \theta$ with $|U^{\model{A}}|=n$, and 
\item if\/ $\model{A}\models\theta$ and $A$ is countably infinite, then $U^{\mathfrak{A}^{\ast}}$ is finite and non-empty.
\end{enumerate}
\end{Lem}

\begin{proof}
Assume $\varphi\in\mathcal{L}(\tau)\setminus \mathcal{L}_{\omega\omega}^{-}(\tau)$ and $\varphi$ depends on finite
$\tau_{0}\subseteq\tau$. Let $\tau_{0}^{\prime}$ be a disjoint copy of $\tau_{0}$, and set
\[
\sigma=\tau_{0}\ \cup\ \tau_{0}^{\prime}\ \cup\ \{<,c_{0},I,G,U,E\},
\]
which results of adding to the vocabulary in the proof of Theorem~\ref{thm1} a unary
predicate symbol $U$ and a binary predicated symbol $E.$ Next, let $\rho\supseteq\alpha$ be finite and consider the sentence $\theta \in\mathcal{L}(\rho)$ which is the conjunction of the theory $\Psi$ introduced in the proof of Theorem~\ref{thm1} plus the following new sentences:

\begin{tabular}
[c]{ll}
7. & $\forall x(Ux\leftrightarrow\exists y(x<y\vee y<x)$ \ \textquotedblleft$U$
is the field of $<$" \ \ \\
8. & $\forall xExx$\\
& $\forall xy\forall\overrightarrow{w}(Exy\rightarrow(\chi(\overrightarrow
{w})\leftrightarrow\chi(\overrightarrow{w}(y/x)))\wedge Ef(\overrightarrow
{w})f(\overrightarrow{w}(y/x)))$, $\ \ \chi,f\in\rho.$\\
& This says that $E$ satisfies the finite list of axioms of equality for the\\
& vocabulary $\rho$, and guarantees that $E$ is the Leibniz congruence
relation \\
& (this is enough by \cite[\S 73 Thm.~41]{kle2})\\
& with respect to $\rho$.\\
9. & $\forall x\lnot(x<x),$ $\forall xyz(x<y\wedge y<z\rightarrow x<z),$\\
& $\forall xy(Ux\wedge Uy\rightarrow x<y\vee y<x\vee Exy)$ $,$\\
& $Uc_{0}\wedge\forall x(Ux\rightarrow x<c_{0}\vee xEc_{0}),$\\
& $\forall xy(Ux\wedge Uy\wedge x<y\rightarrow\exists z(z<y\wedge\forall
w(w<y\rightarrow w<z\vee Ewz))$\\
& These axioms say, with $E$ replacing $=:$ "$<$ \ is a strict linear order of
$U\,\ $\\
& with last element $c_{0}$ and immediate predecesor for non minimal elements"
\end{tabular}

\medskip

\noindent Using~\cite[Lem.\ 4.4 and Prop.\ 4.5]{Casa} and Lemma~\ref{use} as
in the proof of Theorem~\ref{thm1}, for each $n<\omega$, we get a model $\mathfrak{C}=\tuple{\mathfrak{A}_{n}+\mathfrak{B}_{n}^{\prime},<^{\ast},c_{0}^{\ast},I^{\ast},G^{\ast},U^{\ast},E^{\ast}}\models\theta$ where $U^{\model{A}}=\{0,\dots,n\}$, and $E^{\model{A}}$ is true equality.

All that is left to show is that if for a countably infinite structure $\model{A}$ we have $\model{A}\models\theta$, then $U^{\model{A}^{\ast}}$ is finite and non-empty. The first thing to notice is that $<^{\model{A}^{\ast}} $is a strict linear ordering with last element $[c_{0}]$ and immediate predecesors for non miminal elements, because $E$ collapses to true identity
in $\model{A}^{\ast}$. Now, suppose that $U^{\model{A}^{\ast}}$ is infinite,
then we have an infinite descending sequence

\[
\dots<^{\model{A}^{\ast}}[a_{2}]<^{\model{A}^{\ast}}[a_{1}]<^{\model{A}^{\ast
}}[a_{0}]=[c_{0}],
\]
in $U^{\model{A}^{\ast}},$ where $[a_{n+1}]$ is the immediate predecesor of
$[a_{n}].$ But then we have the sequence
\[
\dots<^{\model{A}}a_{2}<^{\model{A}}a_{1}<^{\model{A}}a_{0}\]
in $\model{A}$. Reasoning as in the proof of Theorem~\ref{thm1} (i),
$\model{A}\upharpoonright\tau_{0}\sim_{p}(\model{A}\upharpoonright\tau
_{0}^{\prime})^{-^{\prime}}$ and, since $\model{A}$ is countable,
$\model{A}\upharpoonright\tau_{0}\sim(\model{A}\upharpoonright\tau_{0}^{\prime})^{-^{\prime}}$ but $\model{A}\upharpoonright\tau_{0}\models
\varphi,(\model{A}\upharpoonright\tau_{0}^{\prime})^{-^{\prime}}\models
\lnot\varphi^{\prime}$, contradicting the weak isomorphism property.\qedhere

\end{proof}

\begin{Thm}\label{thm2}
Let $\mathcal{L}$ be an effectively regular abstract logic~\cite[Def.\ II.1.2.4]{barfer} such that  $\mathcal{L}_{\omega \omega}^{-} \leq \mathcal{L}$. Then, $\mathcal{L}$ has the weak isomorphism property, is recursively enumerable for validity, and has the L\"owenheim--Skolem property only if $\mathcal{L} \leq  \mathcal{L}_{\omega \omega}^{-}$.
\end{Thm}

\begin{proof}
Assume for a contradiction that $\mathcal{L}\not \leq \mathcal{L}_{\omega\omega}^{-}.$ Using Vaught's generalization of Trakhtenbrot theorem to $\mathcal{L}_{\omega\omega}^{-}$ \cite{vau}, we obtain a finite purely relational vocabulary $\tau^{\prime}$ such that the set $V_{\text{fin}}\ \subseteq\mathcal{L}_{\omega\omega}^{-}(\tau^{\prime})$ of sentences valid on finite models is not recursively enumerable. Let $\theta \in\mathcal{L}_{\omega\omega}^{-}(\sigma\cup\tau^{\prime})$ where $\sigma$ and $\theta$ are given by Lemma~\ref{rec} (we may obviously assume $\sigma\cap \tau^{\prime}=\emptyset).$ Now we may observe that
\[
\psi\in V_{\text{fin}}\ \ \text{iff}\ \vDash\theta\rightarrow\psi^{U}\!,
\]
where $\psi^{U}$ is the relativization in $\mathcal{L}_{\omega\omega}^{-}$ of $\psi$ to the unary predicate $U$ (which is possible since $\mathcal{L}_{\omega\omega}^{-}$ has the relativization property). If $\psi\in V_{\text{fin}}\ $, then whenever $\model{A}$ is a countably infinite $\sigma\cup\tau^{\prime}$-structure such that $\model{A}\models\theta$ we must have that $U^{\model{A}^{\ast}}$ is finite and non-empty by Lemma~\ref{rec}, thus $\model{A}^{\ast}\models\psi^{U}$, and by the weak isomorphism property, $\model{A}\models\psi^{U}$ as desired (given that $\model{A}\sim \model{A}^{\ast}$). But for any sentence $\chi$ of $\mathcal{L}$, $\vDash\chi$ iff $\chi$ is valid on countably infinite structures: if $\not \vDash \chi$, a countably infinite countermodel for $\chi$ can be found by either applying the L\"{o}wenheim--Skolem property or Lemma~\ref{use} as needed.\footnote{This point is different from the proof of the classical counterpart of the theorem, where equality is available. Obviously, in that setting, from a finite countermodel we cannot simply go to a countably infinite one.} On the other hand, if $\vDash\theta\rightarrow\psi^{U}$ and $\model{A}$ is a $\tau^{\prime}$-model of size $n$, say, we may assume (since $\tau^{\prime}\cap \sigma=\emptyset$) that $\model{A}\cong(\model{A}^{\prime}|U)\upharpoonright \tau^{\prime}$ for a model $\model{A}^{\prime}$ that comes from extending and expanding $\mathfrak{A}$ to a $\tau^{\prime}\cup\sigma$-model of $\theta$ given by (1) in a suitable way. Hence, $\model{A}^{\prime}\models\psi^{U}$ and thus $\model{A}\models\psi.$ Since, by hypothesis, $\mathcal{L}$ is effectively regular and recursively enumerable for validity, we must have then that $V_{\text{fin}}$ is recursively enumerable after all, which is a contradiction.\qedhere

\end{proof}

\begin{Rmk}\emph{
Proper extensions of  $\mathcal{L}_{\omega \omega}^-$ which are recursively enumerable for validity and have the weak isomorphism property are given in Examples~\ref{ex2} and~\ref{ex3} below. Notice that an analogous theorem for the monadic case is trivial because, in the presence of the weak isomorphism property, the effectivity of the logic implies the finite weak dependence property.}
\end{Rmk}

\begin{Rmk}\emph{Other maximality results can be obtained by similar methods to those in this paper. For example, $\mathcal{L}_{\omega \omega}^{-}$ is the maximal logic with the weak isomorphism property, compactness and the so called {\em Tarski union property}. This can be seen by adapting the argument of~\cite[Thm.\ III.2.2.1]{barfer} for $\mathcal{L}_{\omega \omega}$ to the context without identity with the help of~\cite[Prop.\ 2.8]{de}. We conjecture that the $\lambda$-omitting types theorem also provides a characterization of the maximality of $\mathcal{L}_{\omega \omega}^{-}$ (cf.\ \cite{lin2}).}

\end{Rmk}

\section{Extensions of $\mathcal{L}_{\omega \omega}^-$}\label{ext}

In this section, we collect a number of interesting examples of identity-free logics that help answer some  questions posed by our results, e.g.\ \emph{is there a proper extension of $\mathcal{L}_{\omega \omega}^-$  satisfying both the compactness and weak isomorphism properties?}\footnote{The positive answer to this question in Example~\ref{ex3} shows that the L\"owenheim--Skolem property is necessary in Theorem~\ref{thm1}.} Notice that the infinitary logic $\mathcal{L}_{\omega_1\omega}^-$ is an example of an abstract logic with the weak isomorphism and L\"owenheim--Skolem properties, but without compactness.

Our examples will rely on the addition of suitable Lindstr\"om quantifiers which conveniently differ from usual definitions found in the literature. Indeed, adding a Lindstr\"om quantifier to $\mathcal{L}_{\omega \omega}^-$ usually destroys the weak isomorphism property, as is the case with cardinality and cofinality quantifiers. However, each quantifier has a natural  version closed under weak isomorphisms.

\begin{Exm}[The logic $\mathcal{L}_{\omega \omega}^-(Q_{\alpha}^-)$] \label{ex2}\emph{ Consider the Lindstr\"om quantifier $Q_\alpha^-$ defined as: 
\[
 \{ \tuple{A, M, E} \mid M \subseteq A, E  \ \text{equivalence relation on} \ A \ \text{congruent with} \ M,
 \big|\faktor{M}{E}\big|\geq \omega_\alpha  \}.
\]
The satisfaction condition for this operator then is 
\[
\model{A} \models Q_\alpha^- xyz[\f(x), \theta(y,z)]  \ \text{iff} \ \{ \tuple{a, b}
   \in A^2 \mid \model{A} \models \theta[a, b]
   \} \ \text{is an equivalence relation on} \ A,  
\]
\[
\,\,\,\,\,\,\, \,\,\,\,\,\,\,    \,\,\,\,\,\,\,  \,\,\,\,\,\,\,  \model{A} \models \forall xy (\theta(x,y) \rightarrow (\f(x) \rightarrow \f(y))), \ \text{and}
\]
\[
\,\,\,\,\,\,\,  \,\,\,  \,\,\,\,\,\,\, \,\,\,\,\,\,\,  \,\,\,\,\,\,\, \,\,\,\,\,\,\,   \,\,\,\,\,\,\, \,\,\,\,\,\,\,  \,\,\,\,\,\,\,  \,\,\,\,\,\,\,  \,\,\,\,\,\,\,     \
 \big|\faktor{\{a \in A\mid \model{A} \models \f[a] \} }{
 \{ \tuple{a, b}
   \in A^2 \mid \model{A} \models \theta[a, b]
   \}}\big| \geq \omega_\alpha.
\] 
The quantifier $Q_\alpha$ may be recovered by letting $E$ be the real identity relation $=$. 
 }
\end{Exm} 

The first observation we wish to make is that  $Q_1^- $ (seen as a Lindstr\"om quantifier) is closed under weak isomorphisms, i.e.\ if $\tuple{A, M, E}\in Q_1^- $ and $\tuple{A, M, E} \sim \tuple{A', M', E'}$, then  $\tuple{A', M', E'}\in Q_1^- $. To see this, suppose  that $\tuple{A, M, E}\in Q_1^- $ and $R$ is a weak isomorphism from $\tuple{A, M, E}$  onto $\tuple{A', M', E'}$.  $E'$ is an equivalence relation on $A'$ compatible with $M'$ because that fact can be expressed as a formula in $\mathcal{L}_{\omega \omega}^-$. We wish to show then that $R$ induces a bijection $\faktor{M}{E} \longrightarrow \faktor{M'}{E'}$. Consider the relation $R'$ defined as $[x]R'[y]$ iff $xRy$. We wish to show that $R'$ is in fact a bijection. It is obviously surjective since $R$ is. For functionality: assume that $x\in M$, $xRy_1$ and $xRy_2$, then, since $xEx$, we must have that $y_1E'y_2$, which then means that  if  $[x]R'[y_1]$ and  $[x]R'[y_2]$, $[y_1]=[y_2]$.  Injectivity is obtained by an analogous argument in reverse. Hence,  $| \faktor{M'}{E'}| \geq \omega_1$ as desired. 

$\mathcal{L}_{\omega \omega}^-(Q_1^- )$  is clearly more expressive than $\mathcal{L}_{\omega \omega}^-$ since the latter has the L\"owenheim--Skolem property but the former does not (thus, the quantifier $Q_1^-$ is not definable in $\mathcal{L}_{\omega \omega}^-$).
 Recall that a logic $\mathcal{L}$ is said to be \emph{congruence closed} \cite{ma} if, for any $\f \in \mathcal{L}(\tau)$, there is a sentence $\f_E \in \mathcal{L}(\tau \cup \{E\})$ (where $E$ is a new binary predicate)
 such that 
 \[
 (*) \,\,\,\, \faktor{\mathfrak{A}}{\overline{E}} \models \f \ \text{iff} \  \tuple{\mathfrak{A}, \overline{E}} \models \f_E
 \]
 for any structure $\mathfrak{A}$ and any equivalence relation $\overline{E}$ on $A$. We will follow the notation of~\cite{Cai2} in using $q\mathcal{L}$ to denote the congruence closure of a given logic $\mathcal{L}$, obtained by adjoining to $\mathcal{L}$ the sentences defined by $(*)$ as new quantifiers (see \cite{ma}). Then it is not difficult to observe that the logic $\mathcal{L}_{\omega \omega}^-(Q_1^- )$ is contained in the logic (with identity)  $q \mathcal{L}_{\omega \omega}(Q_1)$. By the definition above, $$ \big{|}\faktor{\{a \in A\mid \model{A} \models \f[a] \} }{
 \{ \tuple{a, b}
   \in A^2 \mid \model{A} \models \theta[a, b]
   \}}\big{|} \geq \omega_1$$ can be expressed by the relativized sentence $((Q_1x(x=x))_{\theta})^{\{x \mid \f(x)\}}$. Recall a logic is $(\kappa, \lambda)$-compact if  every set of sentences of cardinality $\leq \kappa$ which has models for each of its subsets of cardinality $<\lambda$, has itself a model.  By \cite[Prop.\ 3.2]{ma}, for any $\mathcal{L}$, if $\mathcal{L}$ is $(\kappa, \lambda)$-compact, so is $q\mathcal{L}$, and hence $q \mathcal{L}_{\omega \omega}(Q_1)$ is $(\omega, \omega)$-compact since $\mathcal{L}_{\omega \omega}(Q_1)$ is, which means that $\mathcal{L}_{\omega \omega}^-(Q_1^- )$  also inherits this property. Once more, by \cite[Prop.\ 3.2]{ma}, since $\mathcal{L}_{\omega \omega}(Q_1)$ is recursively enumerable for validity, $q\mathcal{L}_{\omega \omega}(Q_1)$ is too, and hence, so is the logic $\mathcal{L}_{\omega \omega}^-(Q_1^- )$.

\begin{Exm}[The logic $\mathcal{L}_{\omega \omega}^-(Q^{\text{cf}\omega -})$]\label{ex3} \em

Consider now the following Lindstr\"om quantifier: 
\[
Q^{\text{cf}\omega-} = \{ \tuple{A, M, E} \mid M \subseteq A^2, E  \ \text{is an equivalence relation on} \ A \ \text{congruent with} \ M, 
\]
\[
\faktor{\tuple{A, M}}{E} \ \ \text{is a linear order with cofinality} \ \omega  \}.
\] 

Then, we have that $\model{A} \models  Q^{\text{cf}\omega-} xyzw[\f(x,y), \theta(z, w)]$ iff

\begin{itemize}
\item $\theta^\mathfrak{A}=\{ \tuple{a, b} \in A^2 \mid \model{A} \models \theta[a, b]\} \ \text{is an equivalence relation on} \ A,$

\item $\model{A} \models \forall xy ((\theta(x,y) \wedge \theta(z,w)) \rightarrow (\f(x,z) \rightarrow \f(y, w))),$

\item $\model{A} \models ``\f(x,y) \ \text{is an irreflexive transitive relation}",$

\item $\model{A} \models \All{xy} (\f(x,y) \vee \f(y,x) \vee \theta(x,y)),$ and

\item $\faktor{\tuple{A, \theta^{\mathfrak{A}}} }{ \{ \tuple{a, b} \in A^2 \mid \model{A} \models \theta[a, b] \}} \ \text{has cofinality} \ \omega.$
\end{itemize}
Once more, the quantifier $Q^{\text{cf}\omega}$ can be defined as above by letting $E$ be the true identity relation $=$.

We can show that the quantifier  $Q^{\text{cf}\omega-}$ is closed under weak isomorphisms. Suppose  that $\tuple{A, M, E}\in Q^{\text{cf}\omega-}$  and $R$ is a weak isomorphism from $\tuple{A, M, E}$  onto $\tuple{A', M', E'}$. As in Example~\ref{ex2}, $R'$ defined as $[x]R'[y]$ iff $xRy$ gives a bijection from $\faktor{\tuple{A, M}}{E}$ to $\faktor{\tuple{A', M'}}{E'}$. Furthermore, $R'$ preserves the order: assume that $[x_1]R'[y_1]$, $[x_2]R'[y_2]$ and $\tuple{[x_1], [x_2]} \in M^{\faktor{\tuple{A, M}}{E}}$, so $\tuple{x_1, x_2} \in M$ and, since $x_1Ry_1$ and $x_2Ry_2$, we have $\tuple{y_1, y_2} \in M'$, and thus $\tuple{[y_1], [y_2]} \in M^{'\faktor{\tuple{A', M'}}{E'}}$. Hence, the cofinality of $M^{'\faktor{\tuple{A', M'}}{E'}}$ must be $\omega$ as well.

Shelah's logic $\mathcal{L}_{\omega \omega}(Q^{\text{cf}\omega})$ is the logic $(\infty, \omega)$-compact\footnote{A nice detailed proof can be found in~\cite{Casa3}.} and, by~\cite[Prop.\ 3.2]{ma}, so is $q\mathcal{L}_{\omega \omega}(Q^{\text{cf}\omega})$. But, given that $\mathcal{L}_{\omega \omega}^-(Q^{\text{cf}\omega -})$ is included in $q\mathcal{L}_{\omega \omega}(Q^{\text{cf}\omega})$, the former is also $(\infty, \omega)$-compact. Similarly, $\mathcal{L}_{\omega \omega}^-(Q^{\text{cf}\omega -})$ is recursively enumerable for validity. Moreover, we can observe that $\mathcal{L}_{\omega \omega}^-(Q^{\text{cf}\omega -})$  does not have a L\"owenheim--Skolem theorem. For example, the sentence in the signature $\{ E, <\}$ with two binary relation symbols,

\begin{itemize}
\item[] $\neg Q^{\text{cf}\omega-} xyzw[x<y, E(z, w)]$
\item[] $\wedge \ ``E \ \text{is an equivalence relation}"$

\item[] $\wedge \ \forall xy ((E(x,y) \wedge E(z,w)) \rightarrow (x<z \rightarrow y<w))$

\item[] $\wedge \  ``< \ \text{is an irreflexive transitive relation}"$

\item[] $\wedge \ \All{xy} (x <y \vee y<x \vee E(x,y))$ 
\item[] $\wedge \ \All{x} \Exi{y}(x<y)$

\end{itemize}
has no countable models since it produces in the quotient model an infinite linear order without last element with cofinality $\neq \omega$, and hence $\geq \omega_1$.
\end{Exm} 

Interestingly enough, some known quantifiers can be shown to preserve the weak isomorphism property:

\begin{Exm}[The logic $\mathcal{L}_{\omega \omega}^-(Q^H)$]\emph{Recall the Henkin quantifier $Q^H$ which is defined as follows: 
\[
Q^H = \{\tuple{A, M} \mid M \subseteq A^4, M \supseteq f \times g \ \text{for some} \ f,g \colon A \longrightarrow A\}.
 \]
Then, we have that $\model{A} \models Q^{H} xyzw \f(x,y, z, w)$ iff for some $f,g \colon A \longrightarrow A$ and for each $a,b \in A$, $\model{A} \models \f[a,f(a), b, g(b)]$ iff $\model{A} \models \Exi{f,g}\All{x,y} \f[x,f(x), y, g(y)]$.
}
\end{Exm}

First, we must show that $Q^H$ is closed under weak isomorphisms. Assume then  that $\tuple{A, M, E}\in Q^H$  and $R$ is a weak isomorphism from $\tuple{A, M}$  onto $\tuple{A', M'}$. Then $M \subseteq A^4$, $M \supseteq f \times g$ for some $f, g \colon A \longrightarrow A$. All we need to do now is define $f', g' \colon A' \longrightarrow A'$ such that $M' \supseteq f' \times g'$. Define $f'$ as follows: take any $a_1 \in A'$, we  know  then that $Ra_0a_1$ for some $a_0\in A$, so  let $f'(a_1)$ be some $b_1\in A'$ such that $Rf(a_0)b_1$. Do a similar thing for $g'$. Now, for any $\tuple{a_1, f'(a_1), b_1, g'(b_1)} \in f' \times g'$, there are $a_0, b_0 \in A$ s.t.\ $Ra_0a_1, Rf(a_0)f'(a_1), Rb_0b_1, Rg(b_0)g'(b_1)$, and since $R$ is a weak isomorphism and $\tuple{a_0, f(a_0), b_0, g(b_0)} \in M$ by hypothesis, $\tuple{a_1, f'(a_1), b_1, g'(b_1)} \in M'$, as desired.

Take now the sentence $\f_{\text{inf}} \in \mathcal{L}_{\omega \omega}^-(Q^H)(\tau)$ where $\tau = \{E\}$ and $E$ is binary: 
\[
``E \ \text{is an equivalence relation}"  \wedge  \Exi{z}\Exi{f,g}\All{x,y} (\neg z E f(x) \wedge (f(x) E y \rightarrow g(y) E x))
\]
 Since $Q^H $ is closed under weak isomorphisms, $   \model{A} \sim  \model{A}^* = \faktor{\model{A}}{E^\model{A}}$ in the vocabulary $\tau$, and $\faktor{\model{A}}{E^\model{A}}\models \All{x,y} (xEy \leftrightarrow x=y)$, we have that $\mathfrak{A}\models  \f_{\text{inf}}$ only if $\faktor{\model{A}}{E^\model{A}}\models  \Exi{z}\Exi{f,g}\All{x,y} (z \neq f(x) \wedge (f(x) = y \rightarrow g(y) = x))$. The latter sentence says that $\faktor{\model{A}}{E^\model{A}}$ is infinite.
On the other hand, for a $\tau$-structure $\model{A}$, if $\model{A}\models ``E \ \text{is an equivalence relation}"$   and $\faktor{\model{A}}{E^\model{A}} = \model{A}^*$ is infinite, $\faktor{\model{A}}{E^\model{A}}\models  \Exi{z}\Exi{f,g}\All{x,y} (z \neq f(x) \wedge (f(x) = y \rightarrow g(y) = x))$, so, reversing the previous reasoning, $\model{A} \models \f_{\text{inf}}$. 

Hence, we might consider the following theory $T$ in the vocabulary $\tau$:
\[
\{\neg \f_{\text{inf}}\} \cup \{\Exi{x_0, \dots, x_n}\bigwedge_{i<j \leq n} \neg x_i E x_j \mid  1 \leq n < \omega\} \cup \{``E \ \text{is an equivalence relation}"\}
\]
This theory says that $E$ is an equivalence relation with infinitely many equivalence classes, so for any model $\model{A}\models T$,  $\faktor{\model{A}}{E^\model{A}}$ is infinite and then $\faktor{\model{A}}{E^\model{A}}\models \Exi{z}\Exi{f,g}\All{x,y} (z \neq f(x) \wedge (f(x) = y \rightarrow g(y) = x)$, which is impossible, since $\model{A}\models  \neg \f_{\text{inf}}$. Hence, $T$ has no models. However, $T$ is finitely satisfiable. Thus, compactness fails for the logic $\mathcal{L}_{\omega \omega}^-(Q^H)$, which is then obviously a proper extension of $\mathcal{L}_{\omega \omega}^-$.

To see that $\mathcal{L}_{\omega \omega}^-(Q^H)$ does not have the L\"owenheim--Skolem property consider first the formula $\theta(x,y)$ in the vocabulary $\{E, <\}$:
\begin{itemize}
\item[] $ ``E \ \text{is an equivalence relation congruent with $<$}" $
\item[] $   \exists f, g \forall u, v ((E (u, v) \leftrightarrow E(f(u), g(v))) \wedge (u<x \rightarrow f(v) < y)) $

\item[] $\wedge \  \exists f, g \forall u, v ((E (u, v) \leftrightarrow E(f(u), g(v))) \wedge (u<y \rightarrow f(v) < x))$

\end{itemize}
Now, if $\model{A} \models \theta[a,b]$, since $\model{A}\sim \faktor{\model{A}}{E^\model{A}} $, and given  that $\faktor{\model{A}}{E^\model{A}}\models \All{x,y} (xEy \leftrightarrow x=y)$,

\begin{itemize}
\item[] $\faktor{\model{A}}{E^\model{A}}\models  \exists f, g \forall u, v ((u = v \leftrightarrow f(u) = g(v)) \wedge (u<[a]_E \rightarrow f(u) < [b]_E))$
\item[] $\faktor{\model{A}}{E^\model{A}}\models\exists f, g \forall u, v (( u = v \leftrightarrow f(u) = g(v)) \wedge (u<[b]_E \rightarrow f(u) < [a]_E))$
\end{itemize}
This implies  that  $| \{  z \mid   \faktor{\model{A}}{E^\model{A}}\models z < [a]_E \}| = | \{  z \mid   \faktor{\model{A}}{E^\model{A}}\models z < [b]_E \}| $. Hence, $\theta(x,y)$ is an instance of a H\"artig quantifier in the quotient by $E$.  We can then use this methodology to adapt the typical counterexample for the L\"owenheim--Skolem property for the H\"artig quantifier \cite[Sentence (1.2)]{herr}, axiomatizing infinite linear orderings of successor cardinalities.

\begin{table}[]
\centering
\begin{tabular}{|l|l|l|l|}
\hline
                           Logic    & Compactness & L\"owSko Property & Weak Iso Property            \\ \hline
$\mathcal{L}_{\omega\omega}$ &     $+$       &              $+$                  &         $-$           \\
$\mathcal{L}_{\omega\omega}^-$ &    $+$    &           $+$                    &        $+$                                \\
$\mathcal{L}_{\omega\omega}^-(\{\exists^{\geq n} \mid n\in X \})$ &      $+$        &             $+$                &           $-$                  \\
$\mathcal{L}_{\omega \omega}^-(Q_1^-)$ &    $+$ \ (at least $(\omega, \omega)$)         &          $-$                       &             $+$                  \\
$\mathcal{L}_{\omega \omega}^-(Q_1)$ &        $+$ \ (at least $(\omega, \omega)$)         &            $-$              &        $-$                       \\
$\mathcal{L}_{\omega \omega}^-(Q^{\text{cf}\omega -})$ &        $+$       &           $-$         &                $+$          \\
$\mathcal{L}_{\omega \omega}^-(Q^{\text{cf}\omega })$ &       $+$          &           $-$             &         $-$              \\
$\mathcal{L}_{\omega \omega}^-(Q^H)$ &       $-$         &              $-$       &               $+$              \\ 
$\mathcal{L}_{\omega_1\omega}^-$ &       $-$         &              $+$       &               $+$              \\ 
$\mathcal{L}_{\infty \omega}^-$ &       $-$         &              $-$       &               $+$              \\ \hline
\end{tabular}\vspace{0.5cm}
\caption{Summary of properties of some logics.}
\end{table}

\section{Conclusions}\label{con}

Our work still leaves a number of interesting open questions, including:

\begin{Prob}\label{pmon}
Is there a proper extension of $\mathcal{L}_{\omega \omega}^-$ satisfying both the L\"owenheim--Skolem and compactness properties that is not contained in $\mathcal{L}_{\omega \omega}$?
\end{Prob}

\begin{Prob}\label{prob2}
Is there a compact extension of $\mathcal{L}_{\omega \omega}^-$ which does not remain compact when adding identity to the logic?
\end{Prob}

 \section*{Acknowledgments}
 
 We are grateful to various people who offered useful comments that helped to improve the presentation of the paper, particularly Grigory Olkhovikov and Lloyd Humberstone. Badia was partially supported by the Australian Research Council grant DE220100544. Badia and Noguera were also  supported by the European Union's Marie Sklodowska--Curie grant no.\ 101007627 (MOSAIC project).

\end{document}